\documentclass{proc-l}
\usepackage{latexsym,amsmath,amsfonts,amsthm}

\usepackage[justification=centering,labelformat=empty,labelsep=none]{subfig}
\usepackage{tikz}
\usetikzlibrary{shapes}
\usetikzlibrary{arrows}

\newtheorem{theorem}{Theorem}[section]
\newtheorem{proposition}[theorem]{Proposition}
\newtheorem{lemma}[theorem]{Lemma}
\newtheorem{problem}[theorem]{Problem}

\numberwithin{equation}{section}

\newcommand\ignore[1]{}

\newcommand{\Ex}{\mathbb{E}} 
\newcommand{\Prp}[2]{\mathbb{P}_{#1}\left(#2\right)} 
\newcommand{\Exp}[2]{\mathbb{E}_{#1}\left[#2\right]} 
\newcommand{\Ind}[1]{\textbf{1}_{#1}} 

\def\eps{\varepsilon}

\DeclareMathOperator{\Errop}{Err}
\newcommand{\Err}[2]{\Errop_{#1}(#2)}

\def\Tmix{{ t}_{\rm mix}}
\def\Thit{{T}}
\def\Tprod{{t}_{\rm prod}}
\def\Tces{{t}_{\rm Ces}}

\begin{document}
\title[Set hitting times in Markov chains]{Tight inequalities among set hitting times in {M}arkov chains}

\author[S. Griffiths]{Simon Griffiths}
\address{Instituto Nacional de Matem\'atica Pura e Aplicada (IMPA), Rio de Janeiro, Brazil}
\email{sgriff@impa.br}
\thanks{SG is supported by CNPq Proc.~500016/2010-2.}

\author[R. J. Kang]{Ross J. Kang}
\address{Centrum Wiskunde \& Informatica, Amsterdam, Netherlands}
\email{ross.kang@gmail.com}
\thanks{This work was begun while RJK was at Durham University, supported by EPSRC grant EP/G066604/1.  He is currently supported by a NWO Veni grant.}

\author[R. I. Oliveira]{Roberto Imbuzeiro Oliveira}
\address{Instituto Nacional de Matem\'atica Pura e Aplicada (IMPA), Rio de Janeiro, Brazil}
\email{rimfo@impa.br}
\thanks{RIO is supported by a Bolsa de Produtividade em Pesquisa and a Universal grant from CNPq, Brazil.}

\author[V. Patel]{Viresh Patel}
\address{University of Birmingham, Birmingham, United Kingdom}
\email{viresh.s.patel@gmail.com}
\thanks{This work was begun while VP was at Durham University, supported by EPSRC grant EP/G066604/1. He is currently supported by EPSRC grant EP/J008087/1.}




\begin{abstract}
Given an irreducible discrete-time Markov chain on a finite state space, we consider the largest expected hitting time $T(\alpha)$ of a set of stationary measure at least $\alpha$ for $\alpha\in(0,1)$.
We obtain tight inequalities among the values of $T(\alpha)$ for different choices of $\alpha$.
One consequence is that $T(\alpha) \le T(1/2)/\alpha$ for all $\alpha < 1/2$.
As a corollary we have that, if the chain is lazy in a certain sense as well as reversible, then $T(1/2)$ is equivalent to the chain's mixing time, answering a question of Peres.
We furthermore demonstrate that the inequalities we establish give an almost everywhere pointwise limiting characterisation of possible hitting time functions $T(\alpha)$ over the domain $\alpha\in(0,1/2]$.
\end{abstract}


\maketitle

\section{Introduction}\label{sec:intro}

Hitting times are a classical topic in the theory of finite Markov chains, with connections to mixing times, cover times and electrical network representations~\cite{LPW,Lovasz_RWSurvey}. In this paper, we consider a natural family of extremal problems for maximum expected hitting times.  In contrast to most earlier work on hitting times that considered the maximum expected hitting times of individual states, we focus on hitting {\em sets} of states of at least a given stationary measure.  Informally, we are interested in the following basic question: how much more difficult is it to hit a smaller set than a larger one?
(We note that other, quite different extremal problems about hitting times have been considered, e.g.~\cite{BW90}.)

Following the notation of Levin, Peres and Wilmer~\cite{LPW}, we let a sequence of random v.ariables $X=(X_t)_{t=0}^{\infty}$ denote an irreducible Markov chain with finite state space $\Omega$, transition matrix $P$, and stationary distribution $\pi$.  We denote by $\mu_0$ some initial distribution of the chain and by $\mathbb{P}_{\mu_0}$ the corresponding law.  In the case that $\mu_0 =x$ almost surely, for some $x\in \Omega$, we write $\mathbb{P}_{x}$ for the corresponding law.

Given a subset $A\subseteq \Omega$, the \emph{hitting time} of $A$ is the random variable $\tau_A$ defined as follows:
\[
\tau_A\, \equiv\, \min\{\, t\,:\,X_t\in A\, \}\, .
\]
We shall take particular interest in the maximum expected hitting times of sets of at least a given size.  For $\alpha\in (0,1)$ we define $T(\alpha)=T^P(\alpha)$ as follows:
\[
T(\alpha)\, \equiv\, \max\{ \Ex_{x}[\tau_A]:x\in \Omega, A\subseteq \Omega, \pi(A)\ge \alpha\, \}\, .
\]
In other words, $T(\alpha)=T^P (\alpha)$ is the maximum, over all starting states $X_0=x\in \Omega$ and all sets $A\subseteq \Omega$ of stationary measure at least $\alpha$, of the expected hitting time of $A$ from $x$.

\subsection{The extremal ratio problem}

Note the obvious fact that, given $0 < \alpha<\beta<1$, $T(\alpha)$ is lower bounded by $T(\beta)$ always.  Informally in other words, it is more difficult to hit smaller subsets of the state space.  A natural problem then is to determine how much more difficult this is, i.e.~how large the ratio between $T(\alpha)$ and $T(\beta)$ can become.  We dub this the {\em extremal ratio problem}.

\begin{problem}
Given $0<\alpha<\beta<1$, what is the largest possible value of $T(\alpha)/T(\beta)$ over all irreducible finite Markov chains (on at least two states)?
\end{problem}

\noindent
A first result on this problem was noted by the third author~\cite[Corollary~1.7]{RIO}.

\begin{theorem}\label{thm:RIO}
Fix $0 < \alpha < \beta < 1/2$.  There exists a constant $C_\beta > 0$ such that the following holds.
For any irreducible finite Markov chain, 
\begin{align*}T(\alpha)\,\le\, C_\beta\cdot\frac{T(\beta)}{\alpha}.\end{align*}
\end{theorem}

\noindent
This can be shown via C\`esaro mixing time, specifically as a consequence of an equivalence between $T(\beta)$ for $\beta\in(0,1/2)$ and C\`esaro mixing time for any irreducible chain.  This equivalence, which was recently proved independently by the third author~\cite{RIO} and by Peres and Sousi~\cite{PS}, we will discuss in more detail in Subsection~\ref{sec:mixintro}.

In this paper, we improve upon the above result significantly, without recourse to any results on mixing time.  Our first main result implies that the optimal constant in Theorem~\ref{thm:RIO} is $C_\beta=1$ and that moreover we can include the case $\beta=1/2$. 

\begin{theorem}\label{thm:main}
Fix $0 < \alpha < \beta \le 1/2$.  For any irreducible finite Markov chain, 
\begin{equation}\label{star}
T(\alpha)\, \le\, T(\beta)\, + \, \left(\frac{1}{\alpha}-1\right)\cdot T(1-\beta)\, \leq \, \frac{T(\beta)}{\alpha}.  \tag{$\star$}
\end{equation}
\end{theorem}

\noindent
This bound on $T(\alpha)$ is tight: for any $0 < \alpha < \beta \le 1/2$, there exists an irreducible finite Markov chain for which the three terms in~\eqref{star} are all equal.
Furthermore, $\beta=1/2$ represents a boundary case for Theorem~\ref{thm:main}: for each $\beta > 1/2$, there is a class of irreducible finite Markov chains such that $T(\alpha)/T(\beta)$ is arbitrarily large. Thus we have completely settled the extremal ratio problem.

As an application of Theorem~\ref{thm:main}, we show in Subsection~\ref{sec:mixintro} how mixing time is equivalent to $T(1/2)$ for any irreducible chain, under the added restriction that the chain is lazy in a certain sense as well as reversible; this resolves a problem posed by Peres~\cite{Peres_Open}.

Our strategy for proving Theorem~\ref{thm:main} relies on a simple, but useful proposition, which can be deduced from the ergodic properties of irreducible finite Markov chains. We require the following definitions.  Given two sets $A,B\subseteq \Omega$, we  define
\[
d^+(A,B)\equiv \max_{x\in A}\Exp{x}{\tau_B} 
\quad \text{ and } \quad
d^-(A,B)\equiv\min_{x\in A}\Exp{x}{\tau_B}\, .
\]

\begin{proposition}\label{prop:main} Given an irreducible Markov chain with finite state space $\Omega$ and stationary distribution $\pi$, let $A,C\subseteq \Omega$.  Then
\[
\pi(A)\, \le \, \frac{d^+(A,C)}{d^+(A,C)+d^-(C,A)}\, .
\]
\end{proposition}

\noindent
Both Theorem~\ref{thm:main} and Proposition~\ref{prop:main} are proved in Section~\ref{sec:proofs}.  The examples mentioned after the statement of Theorem~\ref{thm:main} are presented in Section~\ref{sec:examples}. We remark that we have replaced our original proof of Proposition~\ref{prop:main} by a shorter and more elegant argument of Peres and Sousi~\cite{PSpersonal}.

\subsection{The shape problem}

In consideration of Theorem~\ref{thm:main}, it is natural to wonder what form the ratio $T(\alpha)/T(\beta)$ may possibly take.
The second problem we treat is what we call the {\em shape problem}.

\begin{problem}
What is the minimal set of constraints on the possible ``shape'' of the function $T(\alpha)$ over the domain $\alpha\in (0,1/2]$ over irreducible finite Markov chains (on at least two states)?
\end{problem}

\noindent
We show that, in the appropriate limit, the constraints imposed by~\eqref{star} in Theorem~\ref{thm:main} are the only non-trivial constraints on $T(\alpha)$ over the domain $\alpha\in (0,1/2]$. (The trivial constraint is that $T$ must be a decreasing function.) 

We now make this statement rigorous.
Let $\mathcal{F}$ denote the set of decreasing functions $f:(0,1/2]\to \mathbb{R}$ given by $f(\alpha) = T(\alpha)/T(1/2)$ for some irreducible finite Markov chain (on at least two states).  
We also consider limits of such functions.
Let $\mathcal{\overline{F}}$ denote the set of decreasing functions $f:(0,1/2]\to \mathbb{R}$ each of which may be obtained as the almost everywhere~(a.e.)~pointwise limit of functions in $\mathcal{F}$. 
Our second main result is as follows.

\begin{theorem}\label{thm:func} Let $f:(0,1/2]\to \mathbb{R}$ be a decreasing function. Then $f\in\mathcal{\overline{F}}$ if and only if $f(1/2)=1$ and 
\[
f(\alpha)\, \le\, \frac{1}{\alpha} \qquad\text{for all }\alpha\in(0,1/2).
\]
\end{theorem}

\noindent
We prove this by way of a class of chains we call $L$-shaped Markov chains, for which the hitting time functions $T(\alpha)$ can be straightforwardly determined.  We show Theorem~\ref{thm:func} in Section~\ref{sec:examples}.

As it turns out, the constraints given by~\eqref{star} for $0<\alpha < \beta\le 1/2$ are not the only non-trivial constraints on $T(\alpha)$ over the larger domain $\alpha\in(0,1)$.  We demonstrate this in Section~\ref{sec:FRCR}. The shape problem over that larger domain remains an interesting open problem.

\subsection{The connection to mixing times}\label{sec:mixintro}

To put our results into wider context, we now describe the relationship between Theorem~\ref{thm:main} and mixing times. Recall that the (standard) {\em mixing time} of a chain with state space $\Omega$, transition matrix $P$, and stationary distribution $\pi$ is defined as
\[\Tmix^P\equiv \min\left\{t\in \mathbb{N} \,:\, \forall x\in \Omega,\, \forall A\subset \Omega,\, |P^t(x,A) -\pi(A)|\le \frac{1}{4}\right\}.\]
This parameter has various connections to the analysis of MCMC algorithms, to phase transitions in statistical mechanics, and to other pure and applied problems~\cite{LPW}. Aldous~\cite{Aldous_IneqReversible} showed that it is also related to other parameters of the chain, including the following hitting time parameter:
\[\Tprod^P \equiv \max\{\pi(A)\Exp{x}{\tau_A}\,:\,x\in \Omega,\,\emptyset\neq A\subset \Omega\}.\]

\begin{theorem}\label{thm:aldous}
There exists a universal constant $C>0$ such that the following holds. Consider a reversible, irreducible finite Markov chain with transition matrix $P$ that is lazy in the sense that $P_{x\,x}\ge 1/2$ for all $x$ in the state space.
Then
\[\frac{\Tmix^P}{C}\le \Tprod^P\le C\,\Tmix^P.\]
\end{theorem}

\noindent
We remark that Aldous proved Theorem~\ref{thm:aldous} in continuous time, but there are standard methods to transfer his result to discrete time (cf.~\cite[Theorem 20.3]{LPW}).

Aldous's theorem is typically summed up by saying that $\Tmix^P$ and $\Tprod^P$ are ``equivalent up to universal constants", or simply ``equivalent". A similar equivalence was proved for all irreducible finite Markov chains (not necessarily lazy or reversible), with $\Tmix^P$ replaced by {\em C\`{e}saro mixing time}~\cite{ALW}:
\[\Tces^P\equiv \min\left\{t\in \mathbb{N} \,:\, \forall x\in \Omega,\, \forall A\subset \Omega,\, \left|\frac{1}{t}\sum_{s=0}^{t-1}P^s(x,A) -\pi(A)\right|\le \frac{1}{4}\right\}.\]

A drawback of Theorem~\ref{thm:aldous} and its C\`{e}saro mixing version is that it might seem that the mixing time depends on the hitting times of arbitrarily small sets. On the contrary, it transpires that the maximum hitting times of only sets that are large enough is also equivalent to $\Tmix^P$ and $\Tces^P$ (in the analogous senses). The following was proved independently by Peres and Sousi~\cite{PS} and by the third author~\cite{RIO}.

\begin{theorem}\label{thm:largeish}
For each $\alpha\in (0,1/2)$, there exists a constant $c(\alpha)>0$ such that the following holds. Consider a reversible, irreducible finite Markov chain with transition matrix $P$ that is lazy in the sense that $P_{x\,x}\ge 1/2$ for all $x$ in the state space. Then
\[\frac{\Tmix^P}{c(\alpha)}\le \Thit^P(\alpha)\le c(\alpha)\,\Tmix^P.\]
Moreover, for any irreducible finite Markov chain (not necessarily reversible or lazy),
\[\frac{\Tces^P}{c(\alpha)}\le \Thit^P(\alpha)\le c(\alpha)\,\Tces^P.\]
\end{theorem}

\noindent
Note that, together with the C\`esaro mixing time form of Theorem~\ref{thm:aldous}, Theorem~\ref{thm:RIO} now follows.

There is no analogue of Theorem~\ref{thm:largeish} if one allows $\alpha>1/2$: a simple counter-example is given by a random walk on a graph consisting of two large cliques connected by a single edge~\cite{PePersonal}. 
Until now, it was not known whether $\Thit^P(1/2)$ is also equivalent to $\Tmix^P$ and $\Tces^P$. We prove here that this is the case, answering a question of Peres~\cite{Peres_Open}.

\begin{theorem}\label{thm:large}
There exists a universal constant $c>0$ such that the following holds. Consider a reversible, irreducible finite Markov chain with transition matrix $P$ that is lazy in the sense that $P_{x\,x}\ge 1/2$ for all $x$ in the state space. Then
\[\frac{\Tmix^P}{c}\le \Thit^P(1/2)\le c\,\Tmix^P.\]
Moreover, for any irreducible finite Markov chain (not necessarily reversible or lazy),
\[\frac{\Tces^P}{c}\le \Thit^P(1/2)\le c\,\Tces^P.\]
\end{theorem}

\begin{proof}
By Theorem~\ref{thm:aldous} and its C\`esaro mixing time version, it suffices to show that $\Tprod^P$ is equivalent to $\Thit^P(1/2)$. But this is simple: on the one hand,
\[\frac{\Thit^P(1/2)}{2}\le \max\{\pi(A)\Exp{x}{\tau_A}\,:\,x\in \Omega,\,A\subset \Omega,\,\pi(A)\ge 1/2\}\le \Tprod^P,\]
whereas, on the other hand, Theorem~\ref{thm:main} implies that
\[\pi(A)\Exp{x}{\tau_A}\le \pi(A)\,\Thit^P(\pi(A))\le \Thit^P(1/2)\]
if $\pi(A)\le 1/2$, and the fact that $\Thit^P(\cdot)$ is monotone decreasing implies the above inequality also holds if $\pi(A)> 1/2$.
\end{proof}

\subsection{Organization}

The remainder of the article is organised as follows. 
In Section~\ref{sec:proofs}, we prove Theorem~\ref{thm:main}.  In Section~\ref{sec:examples}, we show Theorem~\ref{thm:main} is tight by presenting some two- and three-state Markov chains.  We also prove Theorem~\ref{thm:func} in Section~\ref{sec:examples}. Finally, in Section~\ref{sec:FRCR} we consider the behaviour of $T(\alpha)$ over the larger domain $\alpha\in(0,1)$ and make some concluding remarks.

\section{Proofs for Theorem~\ref{thm:main}}\label{sec:proofs}

We begin by showing that Theorem~\ref{thm:main} is an easy consequence of Proposition~\ref{prop:main}.

\begin{proof}[Proof of Theorem~\ref{thm:main}] Consider an irreducible Markov chain with finite state space $\Omega$ and stationary distribution $\pi$.  Fix a state $x\in \Omega$ and a set $A\subseteq \Omega$ with $\pi(A)\ge \alpha$.  We prove that
\[
\Exp{x}{\tau_A}\, \le\, T(\beta)\, +\, \left(\frac{1}{\alpha}-1\right) \cdot T(1-\beta)\, .
\]
Since $x$ and $A$ are arbitrary, this will suffice to prove the theorem.


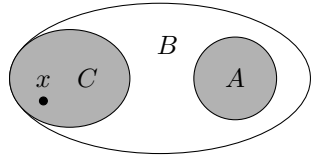
\begin{figure}
\centering
\begin{tikzpicture}[scale=2,draw,nodes={draw,fill=black!0}]
    \draw (0,0) node[ellipse,minimum height=2cm,minimum width=4cm] {};
    \draw (0.05,0.1) node[draw=none,above] {$B$};
    \draw (0.5,0) node[ellipse,minimum height=1.1cm,minimum width=1.1cm,fill=black!30] {$A$};
    \draw (-0.6,0) node[ellipse,minimum height=1.3cm,minimum width=1.6cm,fill=black!30] {\quad\,\,$C$};
    \draw[fill=black] (-0.775,-0.15) circle (0.75pt);
    \draw (-0.775,-0.12) node[draw=none,above,fill=black!30] {$x$};
\end{tikzpicture}
\caption{An illustration of the situation in Theorem~\ref{thm:main}.}
\label{fig:main}
\end{figure}
Define the set $C=C^{\beta}_A$ as follows:
\[
C\, \equiv\, \left\{y\in \Omega: \Ex_y(\tau_A)> \left( \frac{1}{\alpha}-1\right)\cdot T(1-\beta)\right\}\, .
\]
We claim that $\pi(C)< 1-\beta$.  Indeed, if, on the contrary, $\pi(C)$ were at least $1-\beta$, then it would follow that $d^+(A,C)\le T(1-\beta)$ while $d^-(C,A)> (\alpha^{-1}-1) T(1-\beta)$.  This would imply, by Proposition~\ref{prop:main}, that $\pi(A)<\alpha$, a contradiction.  Thus, letting $B\equiv\Omega \setminus C$, we have established that $\pi(B)> \beta$.  Our route from $x$ to $A$ is now clear --- proceed from $x$ to $B$ and then on from $B$ to $A$.  See Figure~\ref{fig:main}.  That is, using the Markovian property of the chain, the expected hitting time of $A$ from $x$ may be bounded by
\[
\Exp{x}{\tau_A}\, \le \, \Exp{x}{\tau_B}\, +\, d^+(B,A) .
\]
Combining the bound $\Exp{x}{\tau_B}\le T(\beta)$ (since $\pi(B)\ge \beta$) with the bound $d^+(B,A) \le (\alpha^{-1}-1)\cdot T(1-\beta)$ (since $B$ is the complement of $C$), we obtain
\[
\Exp{x}{\tau_A}\, \le \, T(\beta)\, +\, \left(\frac{1}{\alpha}-1\right)\cdot T(1-\beta)\, ,
\]
as required.
\end{proof}

All that remains is to prove Proposition~\ref{prop:main}.  As remarked in the introduction, we have replaced our original proof by a shorter and more elegant argument suggested by Peres and Sousi~\cite{PSpersonal}.  Our original proof, which may be obtained at {\tt http://arxiv.org/abs/1209.0039v1}, relied on the ergodic theorem for irreducible Markov chains combined with a martingale concentration inequality.

\begin{proof}[Proof of Proposition~\ref{prop:main}~\cite{PSpersonal}]
Denote the Markov chain by $X$.  Our approach is to define a distribution $\mu$ on $A$ and a distribution $\nu$ on $C$ such that 
\begin{equation}\label{eq:dist}
\pi(A) \Exp{\nu}{\tau_A} \leq (1-\pi(A)) \Exp{\mu}{\tau_C}\, .
\end{equation}
Doing so will complete a proof of the proposition.  Indeed, re-arranging inequality~\eqref{eq:dist}, we obtain
\[
\pi(A)\, \le \frac{\Exp{\mu}{\tau_C}}{\Exp{\mu}{\tau_C}+\Exp{\nu}{\tau_A}}\, \le \, \frac{d^+(A,C)}{d^+(A,C)+d^-(C,A)}\, ,
\]
as required.

We now define the distributions $\mu$ and $\nu$ to satisfy inequality~\eqref{eq:dist}.  Consider an auxiliary Markov chain on $A$ defined by the following transitions: for each $x,y\in A$, let $Q_{xy}$ be the probability that, started from $x$, the first state of $A$ hit by $X$ after time $\tau_C$ is $y$ (i.e.~that $y$ is the first state of $A$ hit after the original chain has reached $C$ from $x$).  Let $\mu$ denote a stationary distribution of this new chain, and let $\nu$ be the hitting distribution on $C$ when the original chain is started from $\mu$, i.e.~$\nu(y)=\Prp{\mu}{X_{\tau_C}=y}$ for each $y\in C$.  

It remains to prove that~\eqref{eq:dist} holds for this choice of $\mu$ and $\nu$.  First observe that, started from the distribution $\mu$, the expected time the chain $X$ spends in $A$ before it reaches $C$ and returns to $A$ is given by $\Exp{\mu}{\tau} \pi(A)$, where $\tau$ denotes the number of steps in such a cycle (from $A$ to $C$ then back to $A$).
This observation is not difficult to verify, but we have included a proof below in Lemma~\ref{lem:splitlemma} of the appendix.
Next, since all visits to $A$ occur before the chain reaches $C$, we have that $\Exp{\mu}{\tau} \pi(A) \le \Exp{\mu}{\tau_C}$. Finally, inequality~\eqref{eq:dist} follows since $\Exp{\mu}{\tau}=\Exp{\mu}{\tau_C}+\Exp{\nu}{\tau_A}$.
\end{proof}

\section{Examples and a proof of Theorem~\ref{thm:func}}\label{sec:examples}

This section is devoted to exhibiting classes of Markov chains which demonstrate that Theorem~\ref{thm:main} is tight, in a few different senses.
\medskip

We first show that equality in~\eqref{star} is attained. For each $0<\alpha<\beta\le 1/2$ we exhibit an irreducible three-state chain with $T(\alpha)=T(\beta)/\alpha$ and hence $T(\alpha)=T(\beta)+(\alpha^{-1}-1)T(\beta) \ge T(\beta)+(\alpha^{-1}-1)T(1-\beta)$, as required.  Consider the three-state chain with transition matrix
\begin{align*}
\begin{pmatrix}
0 \hspace{0.2cm}& 1 \hspace{0.2cm}& 0 \vspace{0.3cm}\\
\frac{\eps}{(1-\alpha-\eps)} \hspace{0.2cm} & 1-\frac{\alpha+\eps}{(1-\alpha-\eps)}\hspace{0.2cm} & \frac{\alpha}{(1-\alpha-\eps)}\vspace{0.3cm} \\
0\hspace{0.2cm} & 1\hspace{0.2cm} & 0
\end{pmatrix}\, ,
\end{align*}
where $0<\eps<\beta-\alpha$.  We note immediately that $(\eps,1-\alpha-\eps, \alpha)$ is the stationary distribution of the chain.  It can be easily checked that $T(\beta)=1$ and $T(\alpha)=1/\alpha$.
\medskip

We next show that the condition $\beta\le 1/2$ in Theorem~\ref{thm:main} is necessary by writing down an irreducible finite chain with $T(\beta)=0$ and $T(\alpha)$ arbitrarily large when $\beta > 1/2$.
Supposing $\beta > 1/2$, let $N$ be an arbitrarily large number and let $\gamma$ be such that $\max\{\alpha,1/2\} < \gamma < \beta$.
Consider the two-state Markov chain with transition matrix
\begin{align*}
\begin{pmatrix}
1-\frac{1}{\gamma N} & \frac{1}{\gamma N} \\
\frac{1}{(1-\gamma)N}   & 1-\frac{1}{(1-\gamma)N} 
\end{pmatrix}.
\end{align*}
The stationary distribution of the chain is $(\gamma, 1-\gamma)$.  It is an exercise to verify that $T(\beta)=0$ and $T(\alpha)\ge(1-\gamma)N$, as desired.
\medskip

We now turn to the proof of Theorem~\ref{thm:func}.  We must prove that each decreasing function $f:(0,1/2]\to \mathbb{R}$ satisfying 
\[
f(\alpha)\, \le\, \frac{1}{\alpha}\, \qquad\text{for all }\alpha\in(0,1/2)
\]
may be obtained as the a.e.~pointwise limit of a sequence of functions $f_1,f_2,\dots$ in $\mathcal{F}$ (i.e.~functions $f_i$ such that $f_i(\alpha) = T^{P_i}(\alpha)/T^{P_i}(1/2)$ for some irreducible finite Markov chain with transition matrix $P_i$).  We first prove this for a certain class of step functions.  Then we consider general functions as limits of these step functions in order to obtain the theorem.

The class of decreasing step functions $f:(0,1/2]\to \mathbb{R}$ we consider are those that may be written in the form
\[f(\alpha)= 1+\sum_{i=1}^{k}\lambda_i\cdot \Ind{\alpha\le \alpha_i},\]
where the $\lambda_i$ and $\alpha_i$ are positive reals satisfying 
\begin{equation}\label{eq:le}
\sum_{j=1}^{i}\lambda_j\le \alpha_i^{-1}-1 \qquad\text{for each }i\in\{1,\dots ,k\} ,
\end{equation}
and $0<\alpha_k<\dots <\alpha_1<1/2$.  We call such a step function \emph{hittable}. We note that if $f$ is a hittable step function then $f(1/2)=1$ and $f(\alpha) \le 1/\alpha$ for all $\alpha\in(0,1/2)$.

Given a hittable step function  $f(\alpha)= 1+\sum_{i=1}^{k}\lambda_i\cdot \Ind{\alpha\le \alpha_i}$, we define the \emph{$\eps$-error set} for $f$ to be the set
\[\Err{f}{\eps} \equiv \bigcup_{i=0}^k[\alpha_i,\alpha_i+\eps],\]
where we interpret $\alpha_0=0$. 

\begin{lemma}\label{lem:hit} Let $f:(0,1/2]\to \mathbb{R}$ be a hittable step function and $\eps>0$.  Then there exists an irreducible finite Markov chain such that $f(\alpha)=T(\alpha)/T(1/2)$ for all $\alpha\in (0,1/2]\setminus \Err{f}{\eps}$.
\end{lemma}

The examples of Markov chains we shall use in the proof of the lemma are all of the same type. 
An \emph{$L$-shaped Markov chain} is a chain whose state space may be labelled $\Omega=\{v_{-1},v_0,v_1,\dots ,v_k\}$ in such a way that the transition matrix of the chain has non-zero entries only at $P_{i\,(i-1)},P_{i\,i},P_{(i-1)\,i},P_{i\,0}$ for $i\in\{0,1,\dots,k\}$. 
Note that $v_0$ is the only state that may be reached directly from a non-adjacent state.  Thus, with the exception of jumps to $v_0$, all transitions are to a neighbour in the sequence $v_{-1},v_0,v_1,\dots ,v_k$.
See Figure~\ref{fig:L}.
In proving Lemma~\ref{lem:hit}, we need only consider $L$-shaped chains. Indeed, it is because the hitting times of such Markov chains are relatively easy to determine that they are suitable for our purposes.  
The following lemma, though somewhat specialised, is exactly what we shall require in our proof of Lemma~\ref{lem:hit}.

\begin{figure}
\centering
\begin{tikzpicture}[->,>=stealth',node distance=1.5cm,scale=2,draw,nodes={circle,draw,fill=black, inner sep=1.5pt}]

  \node (-1) {};
  \node (0) [below of=-1] {};
  \node (1) [right of=0] {};
  \node (2) [right of=1] {};
  \node (3) [right of=2] {};
  \node (4) [right of=3] {};
  \node (5) [right of=4] {};
  \node (6) [right of=5] {};
  \node (7) [right of=6] {};

    \draw (-1) node[fill=none,draw=none,left] {$v_{-1}$};
    \draw (0) node[fill=none,draw=none,below] {$v_0$};
    \draw (1) node[fill=none,draw=none,below] {$v_1$};
    \draw (2) node[fill=none,draw=none,below] {$v_2$};
    \draw (3) node[fill=none,draw=none,below] {$v_3$};
    \draw (4) node[fill=none,draw=none,below] {$v_4$};
    \draw (5) node[fill=none,draw=none,below] {$v_5$};
    \draw (6) node[fill=none,draw=none,below] {$v_6$};
    \draw (7) node[fill=none,draw=none,below] {$v_7$};

\path[every node/.style={font=\sffamily\small}]
(-1) edge [bend right] node {} (0)
     edge [loop right] node {} (-1)
(0)  edge [bend right] node {} (-1)
     edge [loop above] node {} (0)
     edge [bend right] node {} (1)
(1)  edge [bend right] node {} (0)
     edge [loop left] node {} (1)
     edge [bend right] node {} (2)
(2)  edge [bend right] node {} (0)
     edge [bend right] node {} (1)
     edge [loop left] node {} (2)
     edge [bend right] node {} (3)
(3)  edge [bend right] node {} (0)
     edge [bend right] node {} (2)
     edge [loop left] node {} (3)
     edge [bend right] node {} (4)
(4)  edge [bend right] node {} (0)
     edge [bend right] node {} (3)
     edge [loop left] node {} (4)
     edge [bend right] node {} (5)
(5)  edge [bend right] node {} (0)
     edge [bend right] node {} (4)
     edge [loop left] node {} (5)
     edge [bend right] node {} (6)
(6)  edge [bend right] node {} (0)
     edge [bend right] node {} (5)
     edge [loop left] node {} (6)
     edge [bend right] node {} (7)
(7)  edge [bend right] node {} (0)
     edge [bend right] node {} (6)
     edge [loop left] node {} (7);
\end{tikzpicture}
\caption{A depiction of an $L$-shaped Markov chain.}
\label{fig:L}
\end{figure}
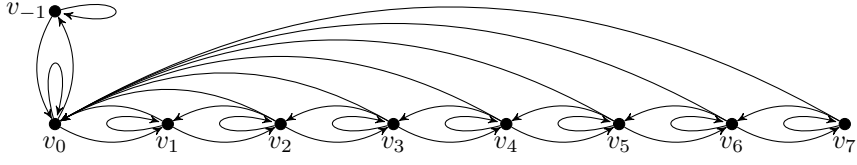

\begin{lemma}\label{lem:T}
Suppose we are given an $L$-shaped Markov chain on state space $\Omega=\{v_{-1},v_0,v_1,\dots ,v_k\}$ with the property that $\Exp{v_j}{\tau_{v_0}}$ is maximised at $j=-1$. If $i\in \{0,\dots ,k\}$ and $\alpha\in (0,1)$ satisfy
\begin{equation}
\label{eq:alphai}
\pi(\{v_{i+1},\dots ,v_k\})+\pi(v_{-1})<\alpha \le \pi(\{v_i,\dots ,v_k\}),
\end{equation}
then 
\[
T(\alpha)=\Exp{v_{-1}}{\tau_{\{v_i,\dots,v_k\}}}=\Exp{v_{-1}}{\tau_{v_i}}\, .
\]
\end{lemma}

\begin{proof} The second equality is obvious since, starting from $v_{-1}$, the chain first arrives in the set $\{v_i,\dots, v_k\}$ at $v_i$.  It is also immediate that $T(\alpha)\ge \Exp{v_{-1}}{\tau_{\{v_i,\dots ,v_k\}}}$, by the definition of $T(\alpha)$ and the assumption that $\pi(\{v_i,\dots ,v_k\}) \ge \alpha$.  Thus all that remains is to prove, for any state $v_j$ and set $A$ with $\pi(A)\ge \alpha$, that $\Exp{v_j}{\tau_A}\le \Exp{v_{-1}}{\tau_{\{v_i,\dots ,v_k\}}}$.

Fix $j\in \{-1,\dots ,k\}$ and a set $A$ with $\pi(A)\ge \alpha$.  Let $i'$ be the minimal non-negative integer for which $v_{i'}\in A$.  The condition on $\alpha$ implies that $i'\le i$.  Now (using the property that $\Exp{v_j}{\tau_{v_0}}$ is maximised at $j=-1$, and the fact that $i'\le i$) we have that
\[
\Exp{v_j}{\tau_A}\le \Exp{v_j}{\tau_{v_0}}+\Exp{v_0}{\tau_{v_{i'}}}\le \Exp{v_{-1}}{\tau_{v_0}}+\Exp{v_0}{\tau_{v_i}}\, .
\]
Since any path from $v_{-1}$ to $v_i$ necessarily passes through $v_0$, the final expression is equal to $\Exp{v_{-1}}{\tau_{\{v_i,\dots ,v_k\}}}$, completing the proof.\end{proof}

The intuition of the above lemma (at least for our intended application) is that if $v_{-1}$ has a very small measure ($\eps$ say) then for almost all values of $\alpha$ (except on a set of measure at most $k\eps$) we know how to express $T(\alpha)$ directly as a hitting time.  This is central to our proof of Lemma~\ref{lem:hit}.

\begin{proof}[Proof of Lemma~\ref{lem:hit}] We shall prove the following assertion: for every hittable step function $f(\alpha)= 1+\sum_{i=1}^{k}\lambda_i\cdot \Ind{\alpha\le \alpha_i}$, every $0<\eps< 1/2-\alpha_1$ and every sufficiently large natural number $N$, there exists an $L$-shaped Markov chain with transition matrix $P$, state space $\Omega=\{v_{-1},v_0,v_1,\dots ,v_k\}$ and stationary measure $\pi$ satisfying 
\begin{itemize}
\item[(i)] $\pi(v_{-1})=\eps$, $\pi(v_0)=1-\alpha_1-\eps$, and $\pi(\{v_i,\dots ,v_k\})=\alpha_i$ for each $i\in\{1,\dots ,k\}$,
\item[(ii)] $\Exp{v_{i}}{\tau_{v_0}}\le N$ for each $i\in\{-1,0,1,\dots,k\}$ with equality if $i = -1$, and
\item[(iii)] $\Exp{v_{i-1}}{\tau_{v_i}}=\lambda_i N$ for each $i\in\{1,\dots,k\}$. 
\end{itemize}

From this assertion Lemma~\ref{lem:hit} easily follows.  Indeed, since $\pi(v_0)=1-\alpha_1 -\eps>1/2$ we have that $T(1/2)$ is precisely the maximum expected hitting time of $v_0$,
and it follows immediately from~(ii) that $T(1/2)= N$.  Given $\alpha\in (0,1/2]\setminus \Err{f}{\eps}$, we shall determine $T(\alpha)$ using Lemma~\ref{lem:T} and condition~(iii).  
In order to apply Lemma~\ref{lem:T}, first notice that condition~(ii) ensures that $\Exp{v_{j}}{\tau_{v_0}}$ is maximised at $j=-1$.  Let $i\in\{1,\dots,k\}$ be smallest such that $\alpha \le \alpha_i$.  Using (i) and the fact that $\alpha\in (0,1/2]\setminus \Err{f}{\eps}$, it is straightforward to verify that~\eqref{eq:alphai} holds in the statement of Lemma~\ref{lem:T}.  Thus, applying Lemma~\ref{lem:T} and using condition~(iii), we have
\begin{align*}
T(\alpha)=\Exp{v_{-1}}{\tau_{v_i}}=\Exp{v_{-1}}{\tau_{v_0}}+\sum_{j=1}^{i}\Exp{v_{j-1}}{\tau_{v_j}}=\left( 1+\sum_{j=1}^{i}\lambda_j\right )N = f(\alpha)T(1/2),
\end{align*}
as required.

We now prove the above assertion by stating explicitly the entries of the transition matrix $P$.  First, we set
\begin{align*}
P_{-1\,0}=\frac{1}{N}&,\quad P_{0\,-1}=\frac{\eps}{(1-\alpha_1-\eps)N},\quad P_{-1\,-1}=1-P_{-1\,0}\\ 
P_{0\,1}=\frac{1-\alpha_1}{(1-\alpha_1-\eps)\lambda_1N}&,\quad P_{1\,0}=\frac{1-\alpha_1-\lambda_1\alpha_2}{(\alpha_1-\alpha_2)\lambda_1N}\quad\text{and}\quad P_{0\,0}=1-P_{0\,-1}-P_{0\,1}. 
\end{align*}
Next, for each $i\in \{2,\dots,k\}$, we set
\begin{align*}
P_{(i-1)\,i}=\frac{1-\alpha_i(1+\sum_{j=1}^{i-1}\lambda_j)}{(\alpha_{i-1}-\alpha_{i})\lambda_iN},\quad P_{i\,(i-1)}=\frac{1-\alpha_i(1+\sum_{j=1}^{i}\lambda_j)}{(\alpha_{i}-\alpha_{i+1})\lambda_iN},\\
P_{i\, 0}=\frac{1}{N}\quad\text{and}\quad P_{i\,i}=1-P_{i\,0}-P_{i\,(i-1)}-P_{i\,(i+1)}.
\end{align*}
Last, we set $P_{1\,1}=1-P_{1\,0}-P_{1\,2}$.  It is routine to verify that each entry in the transition matrix $P$ of our Markov chain is in $[0,1]$ using~\eqref{eq:le}, $0 < \eps < 1/2-\alpha_1$, $0<\alpha_k<\cdots<\alpha_2<\alpha_1$, and a large enough choice of $N$.

Some straightforward calculations confirm that the resulting stationary distribution $\pi$ satisfies condition~(i) above.
Condition~(ii) follows easily from checking that $P_{i\,0} \ge 1/N$ (so that $\Exp{v_i}{\tau_{v_0}} \le N$) for all $i$ and that $\Exp{v_{-1}}{\tau_{v_0}} = N$.
To verify condition~(iii) for each $i\in\{1,\dots,k\}$, we compute the expected hitting time from $v_{i-1}$ to $v_i$ by considering the chain started at $v_{i-1}$ and conditioning on the first step.  We use induction on $i$.  For the base case ($i=1$), we have that
\begin{align*}
\Exp{v_0}{\tau_{v_1}}
& = 1 + P_{0\,0}\Exp{v_0}{\tau_{v_1}} + P_{0\,-1}\Exp{v_{-1}}{\tau_{v_1}}\\
& = 1 + P_{0\,0}\Exp{v_0}{\tau_{v_1}} + P_{0\,-1}(N+\Exp{v_0}{\tau_{v_1}}),
\end{align*}
which implies (after substitution and rearrangement) that $\Exp{v_0}{\tau_{v_1}} = \lambda_1N$.  Next,
\begin{align*}
\Exp{v_1}{\tau_{v_2}}
& = 1 + P_{1\,1}\Exp{v_1}{\tau_{v_2}} + P_{1\, 0}\Exp{v_0}{\tau_{v_2}}\\
& = 1 + P_{1\,1}\Exp{v_1}{\tau_{v_2}} + P_{1\, 0}(\lambda_1N+\Exp{v_1}{\tau_{v_2}}),
\end{align*}
which implies that $\Exp{v_1}{\tau_{v_2}} = \lambda_2N$.  Finally, for $i\in\{3,\dots,k\}$, we have
\begin{align*}
\Exp{v_{i-1}}{\tau_{v_i}}
& = 1 + P_{(i-1)\,(i-2)}\Exp{v_{i-2}}{\tau_{v_{i}}} + P_{(i-1)\,(i-1)}\Exp{v_{i-1}}{\tau_{v_{i}}} + P_{(i-1)\,0}\Exp{v_{0}}{\tau_{v_{i}}}\\
& = 1 + P_{(i-1)\,(i-2)}(\lambda_{i-1}N+\Exp{v_{i-1}}{\tau_{v_i}}) + P_{(i-1)\,(i-1)}\Exp{v_{i-1}}{\tau_{v_{i}}}\\
& \qquad\qquad+ P_{(i-1)\,0}\left(\sum_{j=1}^{i-1}\lambda_jN + \Exp{v_{i-1}}{\tau_{v_i}}\right),
\end{align*}
where the second equality uses the inductive assumption that $\Exp{v_{j-1}}{\tau_{v_j}} = \lambda_jN$ for $j \in \{1,\dots,i-1\}$.
This implies that $\Exp{v_{i-1}}{\tau_{v_i}} = \lambda_iN$, as desired.
\end{proof}

It is now straightforward to deduce Theorem~\ref{thm:func}.

\begin{proof}[Proof of Theorem~\ref{thm:func}] The only if part is an immediate consequence of Theorem~\ref{thm:main}.  Now, fix a decreasing function $f:(0,1/2]\to \mathbb{R}$ with $f(1/2)=1$ that satisfies
$f(\alpha) \le \alpha^{-1}$ for all $\alpha\in(0,1/2)$. Denote by $D=D(f) \subseteq (0,1/2]$ the set of discontinuity points of $f$.  Since $f$ is decreasing, the set $D$ is countable by Froda's theorem\footnote{See {\tt http://en.wikipedia.org/wiki/Froda's\_theorem}.}.
For each positive integer $n$, define the function $f_n:(0,1/2]\to \mathbb{R}$ by
\[
f_n(x)=f(\lceil 2^n x \rceil 2^{-n})\, .
\]
One easily notes that $f_n(x)\to f(x)$ for all $x\in(0,1/2]\setminus D$.

We observe that each $f_n$ is a hittable step function, because it can be written
\begin{align*}
1+\sum_{i=1}^{2^{n-1}-1}  \lambda_i  \Ind{\alpha\le\alpha_i},
\end{align*}
where $\alpha_i=1/2-i2^{-n}$, and $\lambda_i = f(\alpha_i)-f(\alpha_{i-1})$. Condition~\eqref{eq:le} is easily seen to hold since 
\[
1 + \sum_{j=1}^i \lambda_j = 1 + f(\alpha_i ) - f(\alpha_0 ) = f(\alpha_i ) \leq \alpha_i^{-1}. 
\]

To prove the theorem we must find a sequence of functions $g_n\in \mathcal{F}$ such that $g_n(x)\to f(x)$ except on a set of measure zero.  By Lemma~\ref{lem:hit} there exists for each $n$ a function $g_n\in \mathcal{F}$ such that $g_n(x)=f_n(x)$ for all $x\in (0,1/2]\setminus \Err{f_n}{2^{-2n}}$,
where
\begin{align*}
\Err{f_n}{2^{-2n}}=\bigcup_{i=0}^{2^{n-1}-1}\left[\frac{i}{2^n},\frac{i}{2^n}+\frac{1}{2^{2n}}\right].
\end{align*}
We now prove that $g_n(x)\to f(x)$ as $n\to \infty$ for each $x\in (0,1/2]\setminus (D\cup D')$, where $D'$ denotes the set of points that lie in infinitely many intervals of $\Err{f_n}{2^{-2n}}$.  Since $D\cup D'$ has measure zero, this will complete the proof of the theorem.  

To this end, fix $x\in (0,1/2]\setminus (D\cup D')$.  Since $x\not\in D$, we have that $f_n(x)\to f(x)$ as $n\to \infty$.  Furthermore, since $x\not\in D'$, there exists $n_0$ such that 
\[
x\not\in \bigcup_{n\ge n_0}\bigcup_{i=0}^{2^{n-1}-1} \left[\frac{i}{2^n},\frac{i}{2^n}+\frac{1}{2^{2n}}\right],
\] 
and so $g_n(x)=f_n(x)$ for all $n\ge n_0$.  Thus $\lim_{n\to \infty}g_n(x)=\lim_{n\to \infty} f_n(x)=f(x)$, completing the proof of the theorem.\end{proof}

\section{One further result and concluding remarks}\label{sec:FRCR}

For $0 < \alpha < \beta \le 1/2$ we proved the tight inequality $T(\alpha)\le T(\beta)+(\alpha^{-1}-1)T(1-\beta)$ relating hitting times of large enough sets in irreducible finite Markov chains.  Furthermore, we demonstrated that this is the only non-trivial restriction on $T(\alpha)$ as a function over $\alpha\in(0,1/2]$, in the sense made rigorous in Theorem~\ref{thm:func}.  

The most obvious remaining question then is whether there are other non-trivial inequalities relating the values of $T(\alpha)$ for all $\alpha\in (0,1)$.  In one further result, we demonstrate that $T:(0,1)\to \mathbb{R}$ is further constrained.   However, determining the set of all inequalities that hold among the values of $T(\alpha)$ for all $\alpha\in (0,1)$ and thereby giving a characterisation in the spirit of Theorem~\ref{thm:func} of the possible behaviour of $T:(0,1)\to \mathbb{R}$ remains an interesting open problem.

To demonstrate that $T:(0,1)\to\mathbb{R}$ is further constrained it suffices to give a single example of such an additional restriction, which is as follows.

\begin{proposition}\label{prop:another}
Given an irreducible finite Markov chain, suppose $T(0.01)=99.9 T(0.02)$.  Then $T(0.99)\ge 0.1 T(0.02)$.
\end{proposition}

\noindent
We note that this restriction is indeed outside of the class of restrictions imposed by Theorem~\ref{thm:main}.  Writing $T$ for $T(0.02)$, first one can check using Lemma~\ref{lem:hit} that there exist Markov chains satisfying the equality $T(0.01)=99.9T$.  Furthermore, assuming this equality, the application of Theorem~\ref{thm:main} gives that $T(0.01)\le T+99T(0.98)$.  Although this inequality demands that $T(0.98)$ be very close to $T$ --- specifically, $T(0.98)\in [(98.9/99)T,T]$ --- there is no restriction on $T(0.99)$.  Thus Proposition~\ref{prop:another} does indeed represent an additional restriction.  We require the following lemma.

\begin{lemma}\label{lem:ABC}
Given an irreducible Markov chain with finite state space $\Omega$, let $A,B,C\subseteq \Omega$ and
$T$ be a real number such that
\begin{align*}
d^+(\Omega,B)\le T, \quad d^+(\Omega,A\cup C)\le T, \quad d^+(\Omega,A)\le 99.9 T\quad\text{and}\quad d^-(B,A) \ge 98.9T.
\end{align*}
Then 
$d^+(B,C) < 14 T$.
\end{lemma}

\begin{proof}
Let $y \in B$.
Consider running the chain for $10T$ steps and denote by $p_y$ the probability $\Prp{y}{\tau_A\le 10T}$.  The assumptions on the hitting time of $A$ imply that
\[
98.9T\le \Exp{y}{\tau_A}\le 10T + (1-p_y) 99.9T.
\]
Thus $p_y< 0.111< 1/8$.
On the other hand, $\Prp{y}{\tau_{A\cup C}\le 10T}\ge 9/10$ by Markov's inequality, and so $\Prp{y}{\tau_{C}\le 10T}\ge 9/10-1/8 > 3/4$.

We may now bound $d^+(B,C)$ as follows.  Note that, in the event that the chain does not hit $C$ after $10T$ steps, the expected remaining time to hit $C$ may be bounded by $T$ (an upper bound on expected time to return to $B$) plus $d^+(B,C)$ (an upper bound on the expected time to hit $C$ from an element of $B$).  Thus
\[
d^+(B,C) \le 10T\,+ \, \frac{1}{4}(T+d^+(B,C))\, .
\]
It follows that $d^+(B,C)\le 41T/3< 14T$, as required.\end{proof}

We now prove Proposition~\ref{prop:another}.

\begin{proof}[Proof of Proposition~\ref{prop:another}] Let us write $T$ for $T(0.02)$.  Since $T(0.01)=99.9T$ there exists a set $A\subseteq \Omega$ with $\pi(A)\ge 0.01$ and a state $x\in \Omega$ such that $\Exp{x}{\tau_A}=99.9T$.  Define sets 
\[
B'\equiv\{y\in \Omega:\Exp{y}{\tau_A}\le 99T\} \qquad\text{and}\qquad B\equiv\{y\in \Omega:\Exp{y}{\tau_A}\in [98.9T,99T]\}\,.
\]
Arguing as in the proof of Theorem~\ref{thm:main}, one obtains that $\pi(B')\ge 0.98$ --- specifically, if this were not the case, then one would have $d^+(A,\Omega\setminus B')\le T$ and $d^-(\Omega\setminus B',A)>99T$, which contradicts the bound of $\pi(A)\ge 0.01$ using Proposition~\ref{prop:main}.  We now claim that $\pi(B)\ge 0.96$.  Indeed, if on the contrary $\pi(B'\setminus B)$ were greater than $0.02$, then one would obtain $\Exp{x}{\tau_A}< \Exp{x}{\tau_{B'\setminus B}}+98.9T\le 99.9T$, a contradiction.


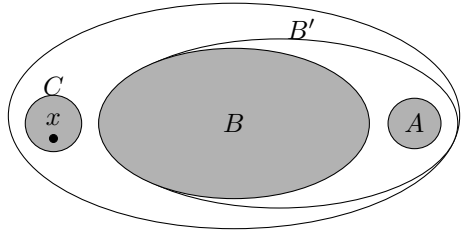
\begin{figure}
\centering
\begin{tikzpicture}[scale=2,draw,nodes={draw,fill=black!0}]
    \draw (0,0) node[ellipse,minimum height=3cm,minimum width=6cm] {};
    \draw (0.3,-0.05) node[ellipse,minimum height=2.25cm,minimum width=4.75cm] {};
    \draw (0.45,0.45) node[fill=none,draw=none,above] {$B'$};
    \draw (0,-0.05) node[ellipse,minimum height=2cm,minimum width=3.6cm,fill=black!30] {$B$};
    \draw (1.2,-0.05) node[ellipse,minimum height=0cm,minimum width=0cm,fill=black!30] {$A$};
    \draw (-1.2,-0.05) node[ellipse,minimum height=0.75cm,minimum width=0.75cm,fill=black!30] {};
    \draw (-1.2,0.08) node[fill=none,draw=none,above] {$C$};
    \draw[fill=black] (-1.2,-0.15) circle (0.75pt);
    \draw (-1.2,-0.12) node[fill=none,draw=none,above] {$x$};
\end{tikzpicture}
\caption{An illustration of the situation in Proposition~\ref{prop:another}.}
\label{fig:another}
\end{figure}
Now, define
\[
C\equiv\{y\in \Omega:\Exp{y}{\tau_A}\ge 99.8T\}.
\]
See Figure~\ref{fig:another}.
We claim that $\pi(C)\le0.01$.  Indeed, if $\pi(C)$ were greater than $0.01$, then the set $A\cup C$ would have stationary measure at least $0.02$, so that $d^+(\Omega,A\cup C) \le T$.  And we would then obtain from Lemma~\ref{lem:ABC} that $d^+(B,C) < 14T$.  On the other hand, $d^-(C,B)\ge 0.8T$ (otherwise, $d^-(C,A)\le d^-(C,B)+d^+(B,A)<0.8T+99T$, which contradicts the definition of $C$). And so, by Proposition~\ref{prop:main}, $\pi(B) < 14T/14.8T =  70/74 < 0.96$, a contradiction.  Thus we have $\pi(\Omega\setminus C)\ge 0.99$ and the inequality $99.9T=\Exp{x}{\tau_A}\le \Exp{x}{\tau_{\Omega\setminus C}}+99.8T$ implies that $\Exp{x}{\tau_{\Omega\setminus C}}\ge 0.1T$.  Therefore $T(0.99)\ge 0.1T$, as required.
\end{proof}

\subsection*{Acknowledgements}

We are grateful to Yuval Peres and Perla Sousi for their kind permission to include their proof of Proposition~\ref{prop:main}.  We also thank the anonymous referee for helpful comments and corrections.

\bibliographystyle{amsplain}
\bibliography{hithithit}

\appendix

\section{A lemma for the proof of Proposition~\ref{prop:main}}

For completeness, we include here a proof of an assertion used in the proof of Proposition~\ref{prop:main}.  This is similar to the proof of Proposition~1.14 of~\cite{LPW}.

\begin{lemma}\label{lem:splitlemma}
Suppose we are given an irreducible Markov chain $X$ with finite state space $\Omega$ and stationary distribution $\pi$.
Let $\mu$ be a distribution on $\Omega$ and $\tau$ be an a.s.~positive stopping time such that 
$\Prp{\mu}{X_\tau = x} = \mu(x)$ for any $x\in\Omega$.
Then for any $S\subset \Omega$ the expected time $X$ spends in $S$ up to time $\tau$ starting from $\mu$ equals $\pi(S)\,\Exp{\mu}{\tau}$.
\end{lemma}

\begin{proof}
For each $x\in\Omega$, define $\widetilde{\pi}(x)$ as the expected time $X$ spends at $x$ up to time $\tau$ when started from $\mu$, i.e.
\[
\widetilde{\pi}(x)\equiv \Exp{\mu}{\sum_{t= 0}^\infty\Ind{\{X_t=x,t<\tau\}}}.
\]
We shall prove that $(\widetilde{\pi}\,P)(x) = \widetilde{\pi}(x)$ for all $x\in \Omega$, which implies that $\widetilde{\pi}$ is a multiple of the (unique) stationary distribution $\pi$. The lemma then follows from 
$\sum_{x\in \Omega}\widetilde{\pi}(x) = \Exp{\mu}{\tau}$, so that $\widetilde{\pi}(x) = \pi(x)\,\Exp{\mu}{\tau}$ for all $x \in\Omega$.

Let us now fix $x\in\Omega$ and compute $(\widetilde{\pi}\,P)(x)$:
$$(\widetilde{\pi}\,P)(x) = \sum_{y\in \Omega}\Exp{\mu}{\sum_{t= 0}^\infty\Ind{\{X_t=y,t<\tau\}}}P_{yx} = \sum_{t= 0}^\infty\sum_{y\in \Omega}\Prp{\mu}{X_t=y,t<\tau}\,P_{yx},$$
since all terms are non-negative. By the Markov property, each term of the double sum equals $\Prp{\mu}{X_t=y,X_{t+1}=x,t<\tau}$, and resolving the inner sum gives
\[
(\widetilde{\pi}\,P)(x) = \sum_{t= 0}^\infty\Prp{\mu}{X_{t+1}=x,t<\tau}.
\]
Now we split each term in the last summation into two parts as follows:
\begin{align*}
\Prp{\mu}{X_{t+1}=x,t<\tau}
& = \Prp{\mu}{X_{t+1}=x,t+1<\tau} + \Prp{\mu}{X_{t+1}=x,t+1=\tau}.
\end{align*}
Summing the first part over $t$ gives
\begin{align*}
\sum_{t= 0}^\infty\Prp{\mu}{X_{t+1}=x,t+1<\tau}
& = \widetilde{\pi}(x)-\Prp{\mu}{X_{0}=x,\tau>0} \\
& = \widetilde{\pi}(x)-\Prp{\mu}{X_{0}=x} = \widetilde{\pi}(x) - \mu(x),
\end{align*}
where the second equality uses the fact that $\tau>0$ a.s. Summing the second part,
\[
\sum_{t= 0}^\infty\Prp{\mu}{X_{t+1}=x,t+1=\tau} = \Prp{\mu}{X_{\tau} = x} = \mu(x),
\]
where we have used $\tau>0$ a.s.~for the first equality and the assumption on $\mu$ for the second. It follows that
\(
(\widetilde{\pi}\,P)(x) = \widetilde{\pi}(x) - \mu(x) + \mu(x) = \widetilde{\pi}(x),
\)
as desired.
\end{proof}

\end{document}